\documentclass{article}%
\usepackage{amssymb}
\usepackage{amsfonts}
\usepackage{amsmath}
\usepackage{amsthm}
\usepackage{graphicx}%
\newtheorem{theorem}{Theorem}

\input{Macros.tex}
\nc{\IR}{\mathbb{R}}
\nc{\IZ}{\mathbb{Z}}

\begin{document}

\title{Farthest points on flat surfaces}
\author{Jo\"{e}l Rouyer
\and Costin V\^{\i}lcu}
\maketitle

%%%%%%%%%%%%%%%%%%%%%%%%%%%%%%%%%%%%%%%%%%%%%%%%%%%%

%%%%%%%%%%%%%%%%%%%%%%%%%%%%%%%%%%%%%%%%%%%%%%%%%%%%%
%\nc{\E}{\mathbb{E}}

\section{Introduction}

This note is an elementary example of the reciprocal influence between Topology and Geometry. A well-known result in this direction is the celebrated
Gauss-Bonnet theorem, which easily implies that the only flat Riemannian surfaces are flat tori and flat Klein bottles. (All \emph{surfaces} considered
here are compact and without boundary.) We consider the distance function from an arbitrary point $p$ on such a surface, and are interested in the set
$F_{p}$ of all \emph{farthest points} (\ie, points at maximal distance) from $p$. Denote by $F$ the (multi-valued) mapping thus defined.

Hugo Steinhauss asked for characterizations of the sets of farthest points on convex surfaces, see Section A35 in \cite{cfg}. In particular, he asked if the
spheres are the only convex surfaces for which $F$ is a single-valued involution. A series of several papers, starting with \cite{z-q}, answered
Steinhauss' questions. The survey \cite{v2}, though a little dated, is still a reference in the field. It should be noted that few examples of explicit
computations of farthest points are known, see \cite{JR-JG} or even \cite{NN}, and the counterexamples to the mentioned conjecture of Steinhauss 
\cite{v-d}, \cite{vz}, \cite{IRV1}, \cite{IV}. On the other hand, no convex polyhedron has $F$ a single-valued involution \cite{JR}.

A motivation for this note lies in the application of Baire Categories theorem to the study of farthest points on Alexandrov surfaces, initiated in
\cite{vz2} and continued in \cite{RV_FPAS}. Roughly speaking, an Alexandrov surface with curvature bounded below by $\kappa$ is a $2$-dimensional
topological manifold endowed with an intrinsic metric which verifies Toponogov's comparison property. For the precise definition and basic
properties of Alexandrov surfaces, see for example \cite{BGP}, \cite{Sh} or, closer to our topic, \cite{ST}. 
It is easy to see that convex surfaces are examples of Alexandrov surfaces with curvature bounded below by $0$. Conversely, any such Alexandrov space is
isometric to a (possibly degenerate) convex surface whenever it is homeomorphic to the sphere.

Consider the space $\mathcal{A}(\kappa)$ of all Alexandrov surfaces, with the topology induced by the Gromov-Hausdorff metric. Two surfaces lie in the same
connected component of $\mathcal{A}(\kappa)$ if and only if they are homeomorphic to each other \cite{RV2}. Hence convex surfaces form a connected
component of $\mathcal{A}(0)$.

By a variant of the Gauss-Bonnet theorem, any topological torus or Klein bottle in $\mathcal{A}(0)$ is (Riemannian and) flat. So two connected
components of $\mathcal{A}(0)$ contain only flat surfaces.

In \cite{RV_FPAS} we extend Tudor Zamfirescu's results in \cite{Z2}, \cite{z-ep}, and show that on \textit{most} (in the sense of Baire categories)
Alexandrov surfaces outside the aforementioned components, most points have a unique farthest point. The cases of the connected components of flat surfaces, namely
the components of $\mathcal{A}(0)$ containing flat tori and flat Klein bottles respectively, are treated in this note with elementary techniques. Somehow
intriguingly, at least at a first glance, flat tori behave differently with this respect from flat Klein bottles, and both classes of flat surfaces behave
differently from all other Alexandrov surfaces.

It is well known that a flat surface can be obtained as the quotient of the Euclidean plane $\mathbb{E}$ by a discrete group of isometries $\Lambda$.
Alternatively, it can be obtained identifying opposite sides of a parallelogram $P$; the correspondence between those constructions is almost obvious. 
What is a little less known is that, for Klein bottles, $P$ can be chosen to be a rectangle without loss of generality, while for tori, its side
lengths and angles may be asked to satisfy a certain inequality. 
For aim of completeness, before determining their farthest points, we reprove the classification of flat surfaces.

On a flat surface $S=\mathbb{E}/\Lambda$, the \emph{cut locus} $C(p)$ of a point $p$ (see any book of Riemannian geometry for the general definition) is
the the set of all points in $S$ joined to $p$ by at least two \emph{segments} (\ie, shortest paths). Hence it is the image under the canonical surjection of
the Voronoi diagram of $p$, seen as a set of points in $\mathbb{E}$. 
In particular, the cut loci on flat surfaces are graphs without extremities, whose edges are segments. 
Moreover, one can easily see that a farthest point from $p$ is necessarily a vertex (ramification point) of $C(p)$. 
Recall that the \emph{Voronoi diagram} of a discrete set of \emph{sites} (\ie, fixed points) $S\subset \mathbb{E}$ is the union of the boundaries of the cells $V_{s}$ ($s\in S$)
consisting of the points of $\mathbb{E}$ that are closer to $s$ than to any other site in $S$.

We shall denote by $F_{p}^{n}$ the set of farthest points from $p$ which are joined to $p$ by exactly $n$ segments
and by $\#S$ the cardinality of the set $S$.

%%%%%%%%%%%%%%%%%%%%%%%%%%%%%%%%%%%%%%%%%%%%%%%%%%%%%

\section{Flat tori}

Denote by $T_{a,b,\alpha}$ the torus obtained (by the usual identification)
from a parallelogram of side lengths $a$ and $b$ and angle $\alpha\in
]0,\frac{\pi}{2}]$ between them.

\begin{theorem}
Any flat torus is isometric to some $T_{a,b,\alpha}$ with $a$, $b$, $\alpha$ satisfying
\begin{equation}
2b\cos\alpha\leq a\leq b\text{.} \label{1}%
\end{equation}

\end{theorem}

\begin{proof}
Let $T$ be a flat torus. The shortest non-contractible closed curve $\gamma_{0}$ is a simple closed geodesic; denote by $a$ its length. 
For $\varepsilon>0$ small enough, the $\varepsilon$-neighborhood $N_{\varepsilon}$ of $\gamma_{0}$ is a flat cylinder. 
The boundary of the strip has two connected components, which are simple closed geodesics on $T$ of length $a$.
Let $\beta_{\varepsilon}^{0},\beta_{\varepsilon}^{1}:\mathbb{R}/a\mathbb{Z} \rightarrow K$ be their arc-length parametrizations. 
Since $T$ has finite area, for greater $\varepsilon$ the strip must self overlap (\ie, 
$\mathrm{Area}\left(  N_{\varepsilon}\right)  = 2a\varepsilon -\mathrm{Area}\left( \mathrm{overlaying }  \right) < 2a\varepsilon$); 
let $c$ be the greatest $\varepsilon$ such that it does not. 
There are distinct points $t_{0}$, $t_{1}\in\mathbb{R}/a\mathbb{Z}$ such that 
$\beta_{c}^{u_{0}}\left( t_{0}\right)  =\beta_{c}^{u_{1}}\left(  t_{1}\right)  $, 
for some $u_{0}$, $u_{1}$ taking values in $\{0,1\}$. Moreover there exists $s=\pm1$ such that
\begin{equation}
\dot{\beta}_{c}^{u_{0}}\left(  t_{0}\right)  =s\dot{\beta}_{c}^{u_{1}}\left(
t_{1}\right)  \text{,} \label{3}%
\end{equation}
for otherwise $\beta_{\varepsilon}^{u_{0}}$ and $\beta_{\varepsilon}^{u_{1}}$
would have intersected for lower $\varepsilon<c$.

Assume first that $u_{0}=u_{1}=0$ (the case $u_{0}=u_{1}=1$ is obviously
similar). Because of (\ref{3}), the lowest $\varepsilon$ such that
$\beta_{\varepsilon}\left(  t\right)  $ has more that one preimage in
$\mathbb{R}/2a\mathbb{Z}$ does not depend on $t$ and equals $c$. Moreover,
since the running direction of $\beta_{c}^{0}$ cannot change all of sudden,
$s=1$. Hence $\beta_{c}^{0}\left(  t_{0}+t\right)  =\beta_{c}^{0}\left(
t_{1}+t\right)  $; it follows that $\left\vert t_{0}-t_{1}\right\vert =a/2$
and a neighborhood of $\beta_{c}^{0}$ is isometric to a M\"{o}bius band, in
contradiction with the fact that $T$ is orientable.

Hence $u_{0}\neq u_{1}$; assume $u_{0}=0$, $u_{1}=1$ and set $\beta
_{-\varepsilon}^{0}=\beta_{\varepsilon}^{1}$, $\varepsilon\in\left[
0,c\right]  $. One can assume without loss of generality that $\beta
_{\varepsilon}^{u}$ is parametrized in such a way that $\tau:\left[
-c,c\right]  \rightarrow T$, $\varepsilon\mapsto\beta_{\varepsilon}^{0}\left(
0\right)  $ is a geodesic normal to $\beta_{0}^{0}$ and $\dot{\beta
}_{\varepsilon}^{0}\left(  0\right)  $ is parallel along $\tau$. By (\ref{3}),
$\beta_{c}^{0}$ and $\beta_{c}^{1}$ are two parametrizations of the same
geodesic, that is
\begin{equation}
\beta_{c}^{1}\left(  t\right)  =\beta_{c}^{0}\left(  \pm t+d\right)  \text{.}
\label{4}%
\end{equation}
Hence $T$ is the cylinder $C\overset{\mathrm{def}}{=}\operatorname{Im}%
\gamma_{0}\times\operatorname{Im}\tau$ whose boundary points are pairwise
identified: $\left(  \gamma_{0}\left(  t\right)  ,\tau\left(  -\varepsilon
\right)  \right)  $ with $\left(  \gamma_{0}\left(  \pm t+d\right)
,\tau\left(  \varepsilon\right)  \right)  $. Since $T$ is orientable, the sign
of $t$ has to be plus. Reversing if necessary the running direction of all
geodesics $\beta_{\varepsilon}^{u}$, one can assume without loss of generality
that $d\in\lbrack0,a/2]$. The line segment of $C$ joining $\beta_{c}%
^{1}\left(  0\right)  $ and $\beta_{c}^{0}\left(  d\right)  $ provides a close
geodesic of $T$ of length $b\overset{\mathrm{def}}{=}\sqrt{d^{2}+4c^{2}}$.
Cutting $C$ along this line provides a parallelogram with side lengths $a$,
$b\geq a$ (because $\gamma_{0}$ is a shortest closed geodesic) and angle
\[
\alpha=\arccos\frac{d}{b}\geq\arccos\frac{a}{2b}\text{.}%
\]

\end{proof}

\begin{theorem}
Let $p$ be a point of the flat torus $T_{a,b,\alpha}$ (with $2b\cos\alpha\leq
a\leq b$). If $\alpha=\pi/2$ then $\#F_{p}=\#F_{p}^{4}=1$, else $\#F_{p}%
=\#F_{p}^{3}=2$.
\end{theorem}

\begin{proof}
$T_{a,b,\alpha}$ is isometric to the standard plane $\mathbb{R}^{2}$
quotiented by the group of translations generated by the vectors
$u\overset{\mathrm{def}}{=}(a,0)$ and $v\overset{\mathrm{def}}{=}(b\cos
\alpha,b\sin\alpha)$. Let $\phi:\mathbb{R}^{2}\rightarrow T_{a,b,\alpha}$ be
the canonical surjection and fix as origin $o=\left(  0,0\right)  $ of
$\mathbb{R}^{2}$ a point in $\phi^{-1}(p)$.

Put $S\overset{\mathrm{def}}{=}(a,0)\mathbb{Z}+(b\cos\alpha,b\sin
\alpha)\mathbb{Z}\subset\mathbb{R}^{2}$. We have $\phi^{-1}(p)=S$ and the cut
locus of $p$ is the image under $\phi$ of the Voronoi diagram of $S$.

If $\alpha=\pi/2$, the Voronoi diagram is a regular rectangular tiling of the
plane, and the conclusion follows immediately. So we may assume, from now on,
that $\alpha<\pi/2$; consequently, the tiles are non-rectangular parallelograms.

\begin{figure}[ptb]
\centering\includegraphics[width=.5\textwidth]{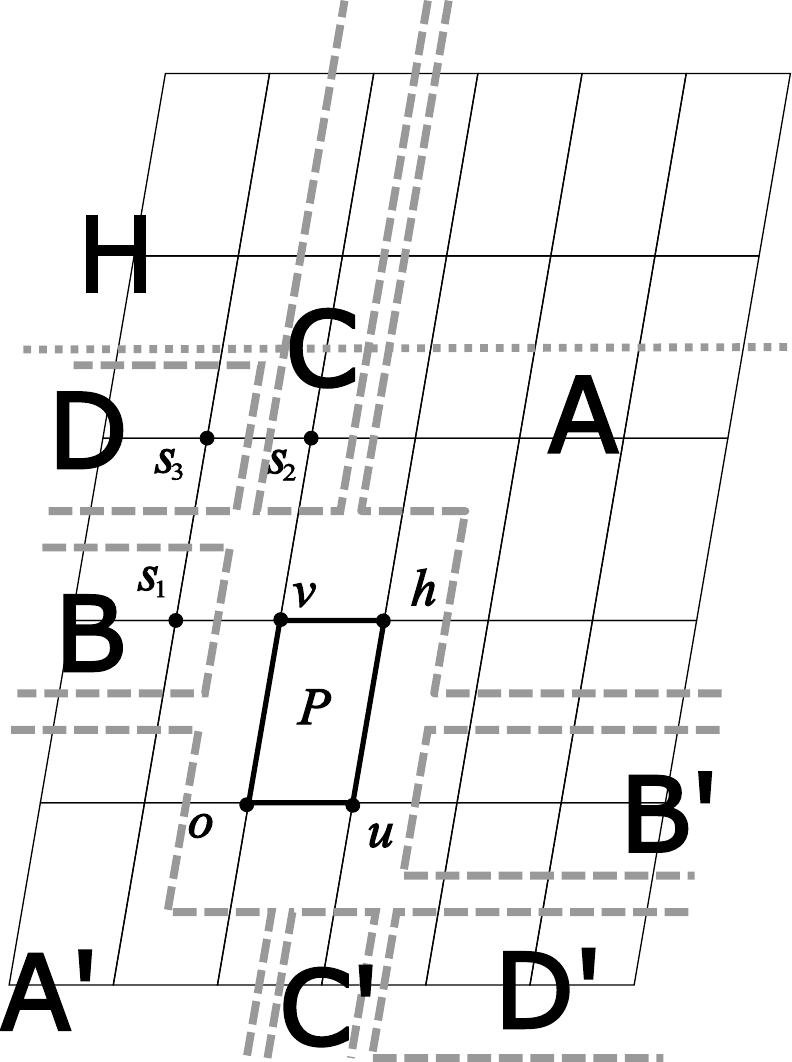} \caption{Tiling of
the plane generated by a flat torus $T_{a,b,\alpha}$.}%
\label{F1}%
\end{figure}

We claim that the restriction to a parallelogram $P$ of the Voronoi diagram of $S$ is the same as the Voronoi diagram of the four vertices of $P$. 
For this, it is sufficient to prove that for any point $x\in P$ and any $s\in S$ which is not a vertex of $P$, there is a vertex $w$ of $P$ such that 
$||s-x||>||w-x||$. 
Choose, for instance, the parallelogram $P$ whose vertices are
$o\overset{\mathrm{def}}{=}(0,0)$, $u\overset{\mathrm{def}}{=}(a,0)$,
$v\overset{\mathrm{def}}{=}(b\cos\alpha,b\sin\alpha)$ and $h\overset
{\mathrm{def}}{=}u+v$. Put $s_{1}\overset{\mathrm{def}}{=}v-u$, $s_{2}%
\overset{\mathrm{def}}{=}2v$ $\ $and $s_{3}=2v-u$.

If $s=mu+nv$ with $m,n\in\mathbb{Z}$, $m,n\geq1$ (zone A in Figure \ref{F1}) then we can choose $w=h$, for the angle $\measuredangle xhs$ is obtuse.
Similarly, if $s=mu+nv$ with $m,n$ negative integers (zone A') then we can choose $w=o$.

Since, by (\ref{1}), $b=||v||<||v-u||$, $o$ (and hence $P$) lies on the right hand side of the bisector of the line-segment $\left[  v,s_{1}\right]  $. 
It follows that we can choose $w=v$ for any $s=s_{1}-ku$, $k\geq0$ (zone B). 
By symmetry, one can choose $w=u$ for $s=u+ku$, $k\geq1$ (zone B').

Still by (\ref{1}), $a=||u||=||h-v||<||h-s_{2}||$, hence $h$ is separated from $s_{2}$ by the mediator $m_{2}$ of the line-segment $\left[  s_{2,}v\right]$. 
It follows that $P$ lies in the lower half-plane bounded by $m_{2}$ and we can choose $w=v$ for any $s=s_{2}+kv$, $k\geq0$ (zone C). 
By symmetry, one can choose $w=u$ for $s=u-kv$, $k\geq1$ (zone C'). 
The condition (\ref{1}) also implies that $P$ lies on the right hand side of the mediator of $\left[s_{2}s_{3}\right]  $, 
whence one can also choose $w=v$ for any $s=s_{3}-ku$, $k\geq0$ (zone D). 
By symmetry one can chose $w=u$ for $s=ku-v$, $k\geq1$ (zone D').

Assume now that $s=ku+k^{\prime}v$, $k^{\prime}\geq3$, $k\in\mathbb{Z}$ (zone H). 
On the one hand, by (\ref{1}), $2b\sin\alpha>b$ and the vertical distance between $P$ and $s$ is at least $b$; hence $d\left(  x,s\right)  >b$ for any $x\in P$. 
On the other hand, the distance from any point of $P$ to the nearest vertex is less than $\frac{a}{2}+\frac{b}{2}$, which is less than $b$. 
Hence the property also holds for those vertices, and by symmetry, for $s=$ $ku-k^{\prime}v$, $k^{\prime}\geq2$, $k\in\mathbb{Z}$.

\begin{figure}[ptb]
\centering\includegraphics[width=.5\textwidth]{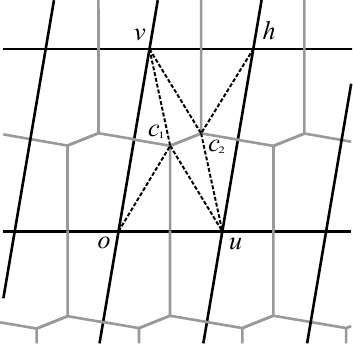} \caption{Voronoi
diagram of the lattice $\mathbb{Z}u+\mathbb{Z}v$.}%
\label{F2}%
\end{figure}

Now, it is easy to check that the vertices of the Voronoi diagram (gray in Figure \ref{F2}) are the circumcenters $c_{1}$ and $c_{2}$ of $\left(o,u,v\right)  $ and $\left(  u,v,h\right)$. 
Obviously the farthest points
from $p=\phi\left(  o\right)  =\phi\left(  u\right)  =\phi\left(  v\right)
=\phi\left(  h\right)  $ are $\phi\left(  c_{1}\right)  $ and $\phi\left(
c_{2}\right)  $, and the segments joining them to $p$ are the images by $\phi$
of the radii (dashed line in Figure \ref{F2}).
\end{proof}

%%%%%%%%%%%%%%%%%%%%%%%%%%%%%%%%%%%%%%%%%%%%%%%%%%%%%

\section{Flat Klein bottles}

Let $K_{a,b}$ be the Klein bottle obtained from an $a\times b$ rectangle by
identifying the points of the $a$-long sides parallelly, and the points of the
$b$-long sides by the symmetry with respect to the center of the rectangle,
$a,b\in\mathbb{R}_{+}^{\ast}$. The direction of the $a$-long sides will be
called \emph{horizontal}, while the direction of the $b$-long sides will be
called \emph{vertical}. The horizontal and vertical geodesics of $K_{a,b}$ are
closed. All horizontal geodesics have length $2a$, excepting the one
corresponding to the sides of the rectangle, and the one corresponding to the
mid-height horizontal segment, which both have length $a$; we call them the
\emph{main geodesics} of $K_{a,b}$.

\begin{theorem}
Any flat Klein bottle is isometric to some $K_{a,b}$.
\end{theorem}

\begin{proof}
Let $K$ be a flat Klein bottle. Consider a closed curve $\gamma$ whose $\varepsilon$-strip is a M\"{o}bius strip, for some $\varepsilon>0$. 
Let $\gamma_{0}$ be a shortest closed curve homotopic to $\gamma$; $\gamma_{0}$ has the same property. 
Moreover, $\gamma_{0}$ is a geodesic; denote by $a$ its length. 
For $\varepsilon>0$ small enough, an $\varepsilon$-strip of $\gamma_0$ is a flat M\"{o}bius strip, 
that is, an $a\times\varepsilon$ rectangle with the $\varepsilon$-long sides identified. 
The boundary of the strip is a closed geodesic of length $2a$, say $\beta_{\varepsilon}:\mathbb{R}/2a\mathbb{Z} \rightarrow K$. 
Since $K$ has finite area, for greater $\varepsilon$, the strip must self overlap; let $b/2$ be the greatest $\varepsilon$ such that it does not. 
There are distinct points $t_{0}$, $t_{1}\in\mathbb{R}/2a\mathbb{Z}$ such that $\beta_{b/2}\left(  t_{0}\right)  =\beta_{b/2}\left(  t_{1}\right) $. 
Moreover
\begin{equation}
\dot{\beta}_{b/2}\left(  t_{0}\right)  =\pm\dot{\beta}_{b/2}\left(
t_{1}\right)  \text{,} \label{1.1}%
\end{equation}
for otherwise $\beta_{\varepsilon}$ would have self-intersected for lower
$\varepsilon$. It follows that the lowest $\varepsilon$ such that
$\beta_{\varepsilon}\left(  t\right)  $ has more that one preimage in
$\mathbb{R}/2a\mathbb{Z}$ does not depend on $t$ and is $b/2$. This defines an
involution $\phi:\mathbb{R}/2a\mathbb{Z\rightarrow R}/2a\mathbb{Z}$, without
fixed points, such that $\beta_{b/2}\left(  t\right)  =\beta_{b/2}\left(
\phi\left(  t\right)  \right)  $. Moreover, from (\ref{1.1}), $\phi\left(
t_{0}+s\right)  =t_{1}\pm s$. Since $\phi$ has no fixed points, the minus case
cannot occur. Since $\phi$ is an involution, $t_{1}-t_{0}=a$, and the proof is complete.
\end{proof}

\bigskip

The circle $\mathbb{R}/a \mathbb{Z}$ acts isometrically on $K_{b}$ by horizontal translation. Other important isometries are the \textquotedblleft
reflections\textquotedblright\ with respect to the main geodesics, and the reflection with respect to a horizontal geodesic which exchanges the main geodesics (\ie,
$\beta_{b/2}$ in the above proof). It follows that any geometrical property attached to a point only depends on its distance $y\in\left[  0,b/4\right]  $
to the nearest main geodesic.

\begin{theorem}
Let $p$ be a point of the flat Klein bottle $K_{a,b}$. Denote by $\sigma$ the
union of its main geodesics and put $\lambda=2d\left(  p,\sigma\right)  /b$. Note that
$\lambda\in\left[  0,1/2\right]  $.

\begin{enumerate}
\item Assume $b<2a$. Then $\#F_{p}=\#F_{p}^{4}=1$ if and only if $\lambda=0$;
otherwise $\#F_{p}=\#F_{p}^{3}=2$.

\item Assume $b=2a$. Then $\#F_{p}=\#F_{p}^{4}=1$ if and only if $\lambda=0$
or $\lambda=1/2$; otherwise $\#F_{p}=\#F_{p}^{3}=2$.

\item Assume $b>2a$ and put $\lambda_{0}=\frac{1}{2}-\sqrt{\frac{1}{4}%
-\frac{a^{2}}{b^{2}}}$.

\begin{enumerate}
\item If $\lambda=0$ then $\#F_{p}=\#F_{p}^{4}=1$.

\item If $0<\lambda<\lambda_{0}$ then $\#F_{p}=\#F_{p}^{3}=2$.

\item If $\lambda=\lambda_{0}$ then $\#F_{p}=\#F_{p}^{4}=1$.

\item If $\lambda_{0}<\lambda<1/2$ then $\#F_{p}=\#F_{p}^{3}=1$.

\item If $\lambda=1/2$ then $\#F_{p}=\#F_{p}^{3}=2$.
\end{enumerate}
\end{enumerate}
\end{theorem}

\begin{proof}
$K_{a,b}$ is isometric to the standard plane $\mathbb{R}^{2}$ quotiented by
the group of affine isometries generated by the translation of vector $(0,b)$
and the glide reflection of vector $(0,a)$ whose axis is the $x$-axis. Let
$\phi:\mathbb{R}^{2}\rightarrow K_{a,b}$ be the canonical surjection. The
main geodesics (dot lines in Figure \ref{F3}) are the image by $\phi$ of the lines of
equations $y=nb/2$, $n\in\mathbb{Z}$. By symmetry, one can assume without loss
of generality that $p=\phi(p_{0})$ with $p_{0}=\left(  0,-\xi\right)  $,
$\xi\in\lbrack0,b/4]$ and set $\lambda=2\xi/b$. Put $S_{0}\overset
{\mathrm{def}}{=}p_{0}+(2a,0)\mathbb{Z}+(0,b)\mathbb{Z}\subset\mathbb{R}^{2}$
and $S_{1}=S_{0}+(a,-2\xi)$ (black and gray dots respectively in Figure
\ref{F3}) . We have
\[
\phi^{-1}(p)=S_{0}\cup S_{1}\overset{\mathrm{def}}{=}S\text{,}%
\]
and the cut locus of $p$ is the image under $\phi$ of the Voronoi diagram of
$S$.

If $\xi=0$, the Voronoi diagram is a regular rectangle tiling of the plane,
and the conclusion follows immediately. So we may assume, from now on, that
$\xi$ is positive.

Line segments from $x\in S_{0}$ to $x+(\pm a,b-2\xi)$ and from $x$ to $x+(\pm
a,-2\xi)$ draw a regular tiling of the plane whose vertices are the points of
$S$ and the tile is kite shaped.

We claim that the restriction to a kite $K$ of the Voronoi diagram of $S$ is
the same as the Voronoi diagram of the four vertices of $K$. For this, it is
sufficient to prove that for any point $x\in K$ and any $s\in S$ which is not
a vertex of $K$, there is a vertex $w$ of $K$ such that $d(s,x)>d(v,x)$.
Choose, for instance, the kite $K$ whose vertices are $p_{0}$, $v\overset
{\mathrm{def}}{=}(0,b-\xi)$ and $h^{\pm}=(\pm a,\xi)$. Let $u=(0,\xi)$ be the
intersection of its diagonals. If $s$ belongs to the positive half of the
ordinates axis, then we can choose $w=v$, for the angle $\measuredangle
xv^{+}s$ is obtuse. Similarly, one can chose $v=p_{0}$ if $s$ belongs to the
negative part of the same axis and $v=h^{\pm}$ if $x$ belongs to the line
$h^{-}h^{+}$. Now assume that $s$ belongs to the open quarter of plane bounded
by half-lines starting at $u$\ through$\ h^{+}$ and $v$. With $q\overset
{\mathrm{def}}{=}(a,b-\xi)$, it is easy to see that on the one hand
$\min\left(  d(x,v),d(x,h^{+})\right)  \leq d(x,q)$, and on the other hand
$\measuredangle xqs>\frac{\pi}{2}$, whence
\[
\min(d(x,v),d(x,h^{+}))<d(x,s)\text{.}%
\]
The cases of the three other quarters of plane are totally similar, so the
claim is proved.

\begin{figure}[ptb]
\centering\includegraphics[width=1.0\textwidth]{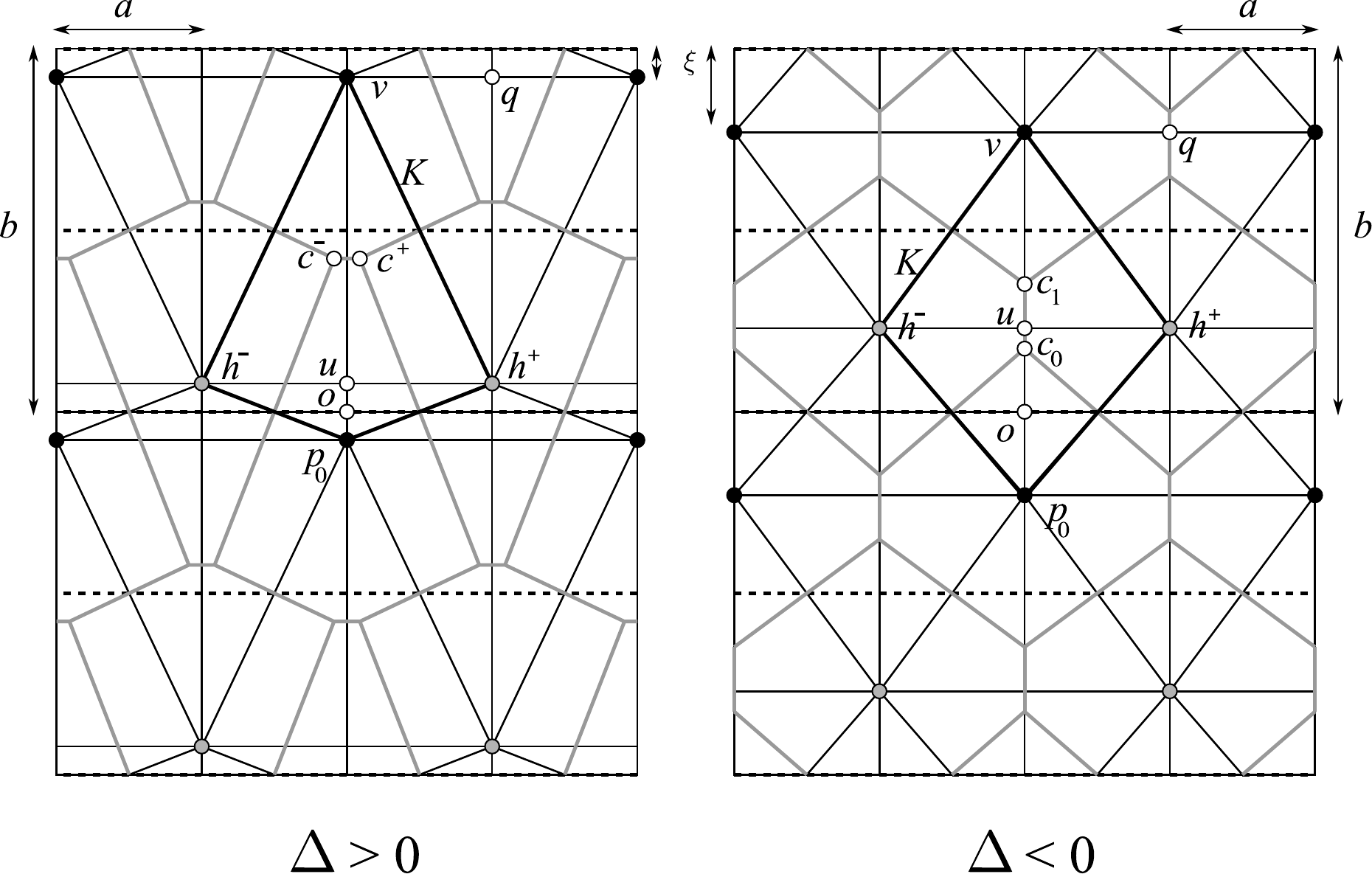} \caption{Tiling
(black lines) of the plane generated by a flat Klein bottle $K_{a,b}$ with the
Voronoi diagram (gray lines) of the preimage of a point $p$.}%
\label{F3}%
\end{figure}

Denote by $\omega\left(  u,v,w\right)  $ the circumcenter of the triangle
determined by $u$, $v$, $w$, where $w$ is a variable point in the plane. Put
$\Delta\overset{\mathrm{def}}{=}a^{2}-b^{2}\lambda(1-\lambda)$. By a
straightforward computation, we have%
\begin{align*}
c^{\pm}\overset{\mathrm{def}}{=}\omega\left(  p_{0},v,h^{\pm}\right)   &
=\left(  \pm\frac{\Delta}{2a},\frac{b}{2}\left(  1-\lambda\right)  \right)
\text{,}\\
c_{0}\overset{\mathrm{def}}{=}\omega\left(  p_{0},h^{+},h^{-}\right)   &
=\left(  0,\frac{a^{2}}{2b\lambda}\right)  \text{,}\\
c_{1}\overset{\mathrm{def}}{=}\omega\left(  v,h^{-},h^{+}\right)   &  =\left(
0,\frac{b^{2}\left(  1-\lambda\right)  ^{2}-\Delta}{2b\left(  1-\lambda
\right)  }\right)  \text{.}%
\end{align*}
It follows that
\begin{align}
d\left(  c_{0},v\right)  ^{2}-d\left(  c_{0},p_{0}\right)  ^{2}  &
=-\frac{\Delta}{\lambda}\text{,}\label{10}\\
d\left(  c^{+},h^{-}\right)  ^{2}-d\left(  c^{+},p_{0}\right)  ^{2}  &
=2\Delta\text{,}\label{11}\\
d(c_{1},v)^{2}-d(c_{1},p_{0})^{2}  &  =\frac{\Delta}{1-\lambda}\text{.}
\label{12}%
\end{align}
We have to discuss three cases, according the sign of $\Delta$.

\textbf{Case} $\Delta>0$. Note that this is the only possible case if $b\leq2a$. From (\ref{10}),
\[
d\left(  c_{0},v\right)  <d\left(  c_{0},p_{0}\right)  =d\left(  c_{0}%
,h^{+}\right)  =d\left(  c_{0},h^{-}\right)  \text{,}%
\]
whence $c_{0}$ belongs to the open Voronoi cell of $v$ and, consequently, is
not a vertex of the Voronoi diagram. Similarly, by (\ref{12}),
\[
d\left(  c_{1},p_{0}\right)  <d\left(  c_{1},v\right)  =d\left(  c_{1}%
,h^{+}\right)  =d\left(  c_{1},h^{-}\right)  \text{,}%
\]
so $c_{1}$ belongs to the open Voronoi cell of $o$. It follows that the
vertices of the restriction to $K$ of the Voronoi diagram of $S$ are precisely
$c^{+}$ and $c^{-}$. The whole Voronoi diagram is shown in Figure \ref{F3} as
a gray line; $\phi\left(  c^{+}\right)  $ and $\phi\left(  c^{-}\right)  $ are
the only points of ramification of the cut-locus of $p$. Due to the vertical
axis symmetry,
\[
d\left(  \phi\left(  c^{+}\right)  ,p\right)  =d\left(  p_{0},c^{+}\right)
=d\left(  p_{0},c^{-}\right)  =d\left(  \phi\left(  c^{-}\right)  ,p\right)
\text{,}%
\]
whence $\phi\left(  c^{+}\right)  $ and $\phi\left(  c^{-}\right)  $ are both
farthest points from $p$. The points $c^{\pm}$ are of degree three in the
Voronoi diagram, whence $\phi\left(  c^{\pm}\right)  \in F_{p}^{3}$.

\textbf{Case} $\Delta<0$. By (\ref{11}),%
\[
d\left(  c^{+},h^{-}\right)  <d\left(  c^{+},p_{0}\right)  =d\left(
c^{+},h^{+}\right)  =d\left(  c^{+},v\right)  \text{,}%
\]
whence $c^{+}$ does not belong to the cell of $h^{+}$, and consequently, it is
not a vertex of the Voronoi diagram. The same argument holds for $c^{-}$.
Moreover, by (\ref{10}) and (\ref{12}),
\begin{align*}
d\left(  c_{0},v\right)   &  >d\left(  c_{0},o\right)  =d\left(  c_{0}%
,h^{+}\right)  =d\left(  c_{0},h^{-}\right)  \text{,}\\
d\left(  c_{1},o\right)   &  >d\left(  c_{1},v\right)  =d\left(  c_{1}%
,h^{+}\right)  =d\left(  c_{1},h^{-}\right)  \text{,}%
\end{align*}
whence $c_{0}$ and $c_{1}$ are the actual vertices of the Voronoi diagram. The
whole Voronoi diagram is shown in Figure \ref{F3}; $\phi\left(  c^{+}\right)
$ and $\phi\left(  c^{-}\right)  $ are the only points of ramification of the
cut-locus of $p$. A simple computation shows that%
\begin{align*}
d\left(  p,\phi\left(  c_{0}\right)  \right)  ^{2}-d\left(  p,\phi\left(
c_{1}\right)  \right)  ^{2}  &  =d\left(  p_{0},c_{0}\right)  ^{2}-d\left(
v,c_{1}\right)  ^{2}\\
&  =\frac{\left(  1-2\lambda\right)  \Delta\left(  a^{2}+b^{2}\left(
1-\lambda\right)  \lambda\right)  }{4b^{2}\left(  1-\lambda\right)
^{2}\lambda^{2}}%
\end{align*}
is non-positive and vanishes if and only if $\lambda=1/2$. It follows that,
for $\lambda<1/2$, $F_{p}=F_{p}^{3}$ $=\left\{  \phi\left(  c_{1}\right)
\right\}  $, and for $\lambda=1/2$, $F_{p}=F_{p}^{3}$ $=\left\{  \phi\left(
c_{0}\right)  ,\phi\left(  c_{1}\right)  \right\}  $.

\textbf{Case} $\Delta=0$. This case occurs only if $b\geq2a$ and corresponds
to a single value of $\lambda\in]0,1/2]$, namely
\[
\lambda=\frac{1}{2}-\sqrt{\frac{1}{4}-\frac{a^{2}}{b^{2}}\text{.}}%
\]
In this case $c^{+}=c^{-}=c_{0}=c_{1}$. Hence $F_{p}^{4}=F_{p}=\{\phi\left(
c_{0}\right)  \}$.
\end{proof}

%%%%%%%%%%%%%%%%%%%%%%%%%%%%%%%%%%%%%%%%%%%%%

\end{document}